\documentclass{amsart}
\usepackage{amsmath,amssymb,amsthm,amsfonts,mathrsfs,textcomp,mathtools,accents,color}
\usepackage{graphicx}
\usepackage{enumerate}
\usepackage{subcaption}

\makeatletter
\@namedef{subjclassname@2020}{%
  \textup{2020} Mathematics Subject Classification}
\makeatother

\usepackage{fullpage}

\newtheorem{theorem}{Theorem}[section]
\newtheorem{corollary}[theorem]{Corollary}
\newtheorem{lemma}[theorem]{Lemma}
\newtheorem{proposition}[theorem]{Proposition}
\theoremstyle{definition}
\newtheorem{definition}[theorem]{Definition}
\theoremstyle{remark}
\newtheorem{rem}[theorem]{Remark}
\theoremstyle{definition}
\newtheorem{example}[theorem]{Example}
\numberwithin{equation}{section}

\newcommand{\set}[1]{\left\{#1\right\}}

\newcommand{\R}{\mathbb R}

\newcommand{\N}{\mathbb N}

\newcommand{\AAA}{\mathcal{A}}

\newcommand{\escrate}[1]{\gamma_{#1}}
\newcommand{\primew}{w_{_P}}
\newcommand{\maxw}{w_{_M}}

\newcommand{\mw}[1]{\mu(w_{#1})}
\newcommand{\m}{\mu}
\begin{document}

\title{Maximal escape rate for shifts}

\author{Claudio Bonanno}
\address{Dipartimento di Matematica, Universit\`a di Pisa, Largo Bruno Pontecorvo 5, 56127 Pisa, Italy}
\email{claudio.bonanno@unipi.it}

\author{Giampaolo Cristadoro}
\address{Dipartimento di Matematica e Applicazioni, Universit\`a di Milano - Bicocca, Via Roberto Cozzi 55, 20125 Milano, Italy}
\email{giampaolo.cristadoro@unimib.it}

\author{Marco Lenci}
\address{Dipartimento di Matematica, Universit\`a di Bologna, Piazza di Porta San Donato 5, 40126 Bologna, Italy, \and Istituto Nazionale di Fisica Nucleare, Sezione di Bologna, Via Irnerio 46, 40126 Bologna, Italy}
\email{marco.lenci@unibo.it}

\date{}

\thanks{The authors are partially supported by 
the PRIN Grant 2017S35EHN ``Regular and stochastic behaviour in dynamical 
systems'' of the Italian Ministry of University 
and Research (MUR), Italy. This research is part of the authors' activity within 
the UMI Group ``DinAmicI'' (\texttt{www.dinamici.org}) and the 
Gruppo Nazionale di Fisica Matematica, INdAM, Italy.}
\begin{abstract}
We consider the shift transformation on the space of infinite sequences over a finite alphabet endowed with the invariant product measure, and examine the presence of a \emph{hole} on the space. The holes we study are specified by the sequences that do not contain a given finite word as initial sub-string. The measure of the set of sequences that do not fall into the hole in the first $n$ iterates of the shift is known to decay exponentially with $n$, and its exponential rate is called \emph{escape rate}. In this paper we provide a complete characterization of the holes with maximal escape rate. In particular we show that, contrary to the case of equiprobable symbols, ordering the holes by their escape rate corresponds to neither the order by their measure nor by the length of the shortest periodic orbit they contain. Finally, we adapt our technique to the case of shifts endowed with Markov measures, where preliminary results show that a more intricate situation is to be expected.
\end{abstract}
\maketitle

\section{Introduction} \label{sec:intro}

Given a measure-preserving transformation $T$ of a probability space $(X,\mu)$, a \emph{hole} of the associated dynamical system $(X,\mu,T)$ is a measurable set $H\subset X$ with $\mu(H)>0$ such that when $T^k(x) \in H$ the orbit of $x$ escapes from $X$. The presence of a hole is often realized by modifying $T$ so that it is not defined on $H$ (another option is to define the maps as the identity inside the hole).
The investigation of properties of open systems is more than forty years old; the first attempts can be traced back to the late '70s. See  \cite{PY79,FKMP95} and references therein for the first approaches to the problem.
This kind of system has  often been used to model different situations of interest for the  physics community. A recent review dealing with such applications can be found in \cite{APT13}.

Let us introduce the basic properties of the open systems we consider. When the probability measure $\mu$ is ergodic, almost every orbit enters $H$ at some finite time, hence almost every orbit escapes.
Let
\[
S_n:= \set{x\in X \, :\, T^i(x)\not\in H\, , \, \forall\, i=0,\dots,n}
\]
be the set of points which do not escape up to time $n$, and define the \emph{survival probability at time $n$} as $p_n:= \mu(S_n)$. The sequence $p_n$ is decreasing and vanishing as $n\to \infty$, and we consider its exponential rate of convergence.

\begin{definition} \label{surv-prob} The \emph{escape rates} $\gamma_H^\pm$ of the set $H$ are the exponential rates of convergence to 0 of the survival probability, namely
\[
\gamma_H^-:= \liminf_{n\to \infty}\, \frac{-\log p_n}{n}\, \quad \text{and} \quad \gamma_H^+:= \limsup_{n\to \infty}\, \frac{-\log p_n}{n}.
\]
\end{definition}

We consider only $\escrate{H}:=\gamma_H^-$,  since it can be studied by the classical method of  generating functions, and we refer to it as the \emph{escape rate of $H$}.

When $T$ has exponential decay of correlations, a finite non-zero escape rate is to be expected for generic holes
(see \cite{DY06,DWY12} for precise conditions on $T$ and a more extensive discussion on this topic).
A variety of systems satisfy this property: paradigmatic examples are expanding maps, Anosov diffeomorphisms, and dispersing billiards \cite{CMT98, CdB02,LM03,DWY10} (but many more cases are studied in the literature). Other aspects that have been investigated are the relations with other statistical properties of  dynamical systems and thermodynamic methods (see \cite{DT17a,BDT18}). Finally, there has been some recent interest also in open systems exhibiting sub-exponential rates of escape \cite{FMS11,MK,BC,DT17b}.

Among the different properties of open systems, much effort has been devoted to the study of the escape rate as a function of certain parameters, such as the size and the position of the hole.
In this context, a recurrent question in the literature concerns the identification of the hole with maximal escape rate among holes in a given family (e.g., in the set of holes with the same measure).
It turns out that a peculiar role is played by the structure of the periodic orbits of the system: as an example, it has been shown that in certain systems the different escape rates for holes with the same measure can be ordered according to the shortest period of the periodic points contained in the hole \cite{BY}. Moreover, using a perturbative approach, it has been shown in \cite{KL09} that the escape rate for holes shrinking to a periodic point has a non-trivial dependence on the stability of the limiting orbit.
While these results provide an answer to the question of where to place a hole to achieve maximal escape for some specific situations, such as the small-hole limit, a complete understanding of the generic case is not yet available, to our knowledge.
This work provides a complete answer to this question for full shifts on infinite sequences over a finite number of symbols. It is well known that such a symbolic dynamical system is isomorphic to an appropriate piecewise linear map on the interval $[0,1]$ with full branches. There are many other examples of systems isomorphic to a full shift, e.g., the logistic map $T(x)=4x(1-x)$ on $[0,1]$. Thus our results apply to these isomorphic systems as well.

More in detail, in Theorem \ref{thm-elem-2}, we show that in order to find the maximal escape rate for a fixed length of the forbidden word, it is enough to consider two specific holes. This characterization is more precise in the case of full shifts over two symbols, where, based on the probability of the most probable symbol, we are able to determine which of the two holes achieves maximal escape rate (Theorem \ref{thm:biggest-two-symb}) and to estimate this rate from below and from above (Corollary \ref{cor:stime}).

\section{The setting} \label{sec:escape}
Consider the symbolic dynamical system $(\AAA^{\N}, \mu,\sigma)$ defined by the left shift transformation $\sigma$ acting on the space $\AAA^{\N}$ of semi-infinite words with symbols from a finite alphabet  $\AAA=\{a_1,a_2,\dots,a_A\}$, endowed with a product probability measure $\m$.  The  measure $\m$ is determined by a probability vector $\{p_{a_1},\dots,p_{a_A}\}$ such that  $p_{a_j}>0$ for all $j$.

Since $\sigma$ preserves $\mu$ and is ergodic, we can study the escape rates for holes in $\AAA^{\N}$. In this context, cylinders are the natural and often studied choice for  holes. This corresponds to fixing a finite word $w = (w_0\, w_1\dots w_{r-1}) \in \AAA^r$ and letting the hole $H$  be the set of all infinite words in $\AAA^{\N}$ containing $w$ as the initial sub-word. Throughout, we denote a hole of this kind by the finite word specifying it. Also, the \emph{length of a hole} denotes the length of the corresponding word.

By the previous construction we are led to use combinatorial arguments in our approach to the escape rate. We first recall some basic notions from \cite{FS}, starting with the definition of \emph{weighted autocorrelation polynomial} of a word.

\begin{definition}\label{count-fun}
Let $w \in \AAA^*:=\cup_{n=1}^\infty \AAA^n$ be a finite word and denote by $|w|$ its length. For any letter $a$ in the alphabet $\AAA$ we define the \emph{number of occurrences of $a$ in a sub-word of $w$} as
\[
N_w(a,k,n) := \left\{ \begin{array}{ll}
\# \set{i \in [k,n-1]\, :\, w_i =a}\, ,& \text{for }\, 0\le k< n\le |w|\, ;\\[0.2cm]
0\, , &  \text{for }\, k=n\, .
\end{array} \right.
\]
For simplicity we use the notation $N_w(a):= N_w(a,0,|w|)$.
\end{definition}

\begin{definition}\label{wap}
Let $\AAA=\{a_1,\dots,a_A\}$ and $w\in \AAA^n$. The \emph{autocorrelation vector} $c=(c_0,\dots,c_{n-1})$ of $w$ is defined by setting
\[
c_i=\left\{ \begin{array}{ll} 1\, , & \text{if } (w_i\, w_{i+1}\, \dots\, w_{n-1}) = (w_0\, w_1\, \dots\, w_{n-1-i})\, ;\\[0.2cm] 0\, , & \text{otherwise.} \end{array} \right.
\]
The \emph{weighted autocorrelation polynomial} of $w$ is a polynomial in $A+1$ variables, $x_{a_1},x_{a_2},\dots,x_{a_A}$ and $z$, given by
\[
c_w(x_{a_1},x_{a_2},\dots,x_{a_A},z) := \sum_{j=0}^{n-1}\, c_j\, \Big( \prod_{a_i\in \AAA}\, (x_{a_i})^{N_w(a_i,n-j,n)} \Big)\, z^j.
\]
\end{definition}
By definition, the weighted autocorrelation polynomial has non-negative coefficients and in particular $c_0=1$.

\begin{definition}\label{prime-hole}
A hole $w$ of finite length is called \emph{prime} if its autocorrelation vector is $c=(1,0,\ldots,0)$, whence $c_w(x_{a_1},x_{a_2},\dots,x_{a_A},z)=1$.
\end{definition}

We now show that the escape rate of a hole $w$ is the logarithm of a root of a polynomial depending on the measure of the set $\mw{}$ and on the weighted autocorrelation polynomial of $w$. A similar result can be found in \cite{agar,agar2}. The proof of the proposition is in Appendix \ref{app:proof-thm-iid} and is based on the notion of the generating function of the survival probability of the hole.

\begin{proposition}\label{prop:escape-rate}
The escape rate $\escrate{w}$ of a hole $w$ of length $r$ is given by
\[
\escrate{w} = \log \, z_0,
\]
where $z_0$ is the smallest positive root of the polynomial
\begin{equation} \label{pol-tau}
\tau_{w}(z) := \mw{}\, z^r + (1- z)\, c_w(p_{a_1},p_{a_2},\dots,p_{a_A},z),
\end{equation}
$\mw{} = \prod_{a_j\in \AAA}\, (p_{a_j})^{N_w(a_j)}$ is the measure of the hole $w$, and $c_w(x_{a_1},x_{a_2},\dots,x_{a_A},z)$ is the weighted autocorrelation polynomial of $w$.
\end{proposition}

Note that,  for all $r\ge 1$ and all $w\in \AAA^r$ we have $\tau_{w}(z)>0$ for all $z\in [0,1]$, and in particular $z_0>1$, so that $\escrate{w}>0$.

Among the different  polynomials $\tau_w$ for the different possible words $w$, two families stand out as particularly relevant for what follows: the polynomials for prime words and the polynomials for words that are repetitions of a single symbol, as for example  $w=(aaa\dots a)$.

For a prime hole $w$ of length $r$ and measure $\m$,  it follows immediately  from  \eqref{pol-tau}  that
\begin{equation} \label{pol-prime}
\tau_{w}(z)= \m z^r -z+1.
\end{equation}

For a hole of the form $w=(aaa\dots a)$ of length $r$, we start computing explicitly the weighted autocorrelation polynomial
\[
c_{w}(x_{a},x_{a_2},\dots,x_{a_A},z) = \sum_{j=0}^{r-1}\, x_a^j\, z^j = \frac{x_a^r\, z^r -1}{x_a\, z-1},
\]
where we are using the notation of Definition \ref{wap} with the convention $a=a_1$ for simplicity. Then
\begin{eqnarray}
\tau_{w}(z)&=& p_a^r\, z^r +(1-z)\, \frac{p_a^r\, z^r -1}{p_a\, z-1} \nonumber\\
&=& \frac{p_a^r(1-p_a)\, z^{r+1}-z+1}{1-p_a z}\label{pol-constant}.
\end{eqnarray}

Note that the numerator in \eqref{pol-constant} is a polynomial belonging to the previous family \eqref{pol-prime}; more precisely, it is the polynomial  of a prime hole of length $r+1$ (this connection was already found in \cite{CDK}). Moreover, both the numerator and the denominator of \eqref{pol-constant} vanish at $\frac{1}{p_a}$ (this will be useful below).

In the next lemma, we show some elementary properties of the family of polynomials in \eqref{pol-prime} that will be useful in the derivation of our main results in next section.

\begin{lemma}\label{lemma-tau}
Let us consider the family of polynomials
\[
f_m(z) = m\, z^r -z+1
\]
for a fixed $r\ge 2$ and $m\in \R^+$. Then:
\begin{enumerate}[(i)]
\item $f_m(z)$ is convex for all $m\in \R^+$ in the set $(0,+\infty)$.
\item Letting $m^*_r:= \frac 1r (1-\frac 1r)^{r-1}$ and $z^*_r(m):= ( r\, m)^{-\frac{1}{r-1}}$, one has
\[
f_m(z^*_r(m)) \left\{ \begin{array}{ll} <0 \, , & \text{if } m< m^*_r;\\[0.2cm] =0\, , & \text{if } m=m^*_r;\\[0.2cm] >0\, , & \text{if } m>m^*_r; \end{array} \right.
\]
and $ z^*_r(m^*_r)=\frac{r}{r-1}$.
\item The polynomial $f_m(z)$ has two positive roots for $m\in (0,m^*_r)$, one positive root for $m=m^*_r$, no positive roots for $m>m^*_r$.
\end{enumerate}
\end{lemma}

\begin{proof}
(i). Obvious.\\
(ii) and (iii). Notice that $f_m'(z)=0$ if and only if $z=z^*_r(m)$, hence $z^*_r(m)$ is a point of local minimum. The sign of $f_m(z^*_r(m))$ is a computation.
\end{proof}

\section{The hole with maximal escape rate} \label{sec:bigger}

In this section we show how to determine the hole with maximal escape rate. That is, we study how the escape rate varies among the holes of length $r$ and we determine which one has the maximal escape rate.  We start this section by discussing one of the simplest examples in our setting, that is, the full shift over two symbols only: this example suffices to show the difficulties that arise when trying to order all holes of fixed length $r$ by their escape rate. Indeed, the first discussion of this problem in the context of dynamical systems can be found in \cite{BY}: Bunimovich and Yurchenko showed that if the symbols of $\AAA$ are equiprobable, that is, $p_{a_j}=\frac 1A$ for all $j=1,\dots,A$, all the holes can be ordered according to their escape rate just by looking at the length of the shortest periodic pattern they contain. In the following example we show that, contrary to the above simpler case, if the symbols have different probabilities, looking only at the periodic patterns in a hole is not enough to determine the hole with maximal escape rate.
Our main result proves that the measure of a hole, its weighted autocorrelation polynomial and the probability of the most probable symbol are the essential ingredients to identify the maximal escape rate for all possible shifts.

\begin{example}\label{two-symbols}
Let $\AAA=\set{a,b}$ and $p=p_a$, $q=p_b$, with $p+q=1$. We restrict to $p\ge q$  (that is $p\in [\frac 12,1]$) as the other cases can be recovered by interchanging the symbols of the alphabet.
We use  Proposition \ref{prop:escape-rate}  to compute explicitly the escape rates of all holes with length $r\le 4$: for these short lengths, the maximal escape rate can be identified by an explicit computation.
\begin{figure}[htp]
\begin{center}
\includegraphics[width=10cm,keepaspectratio]{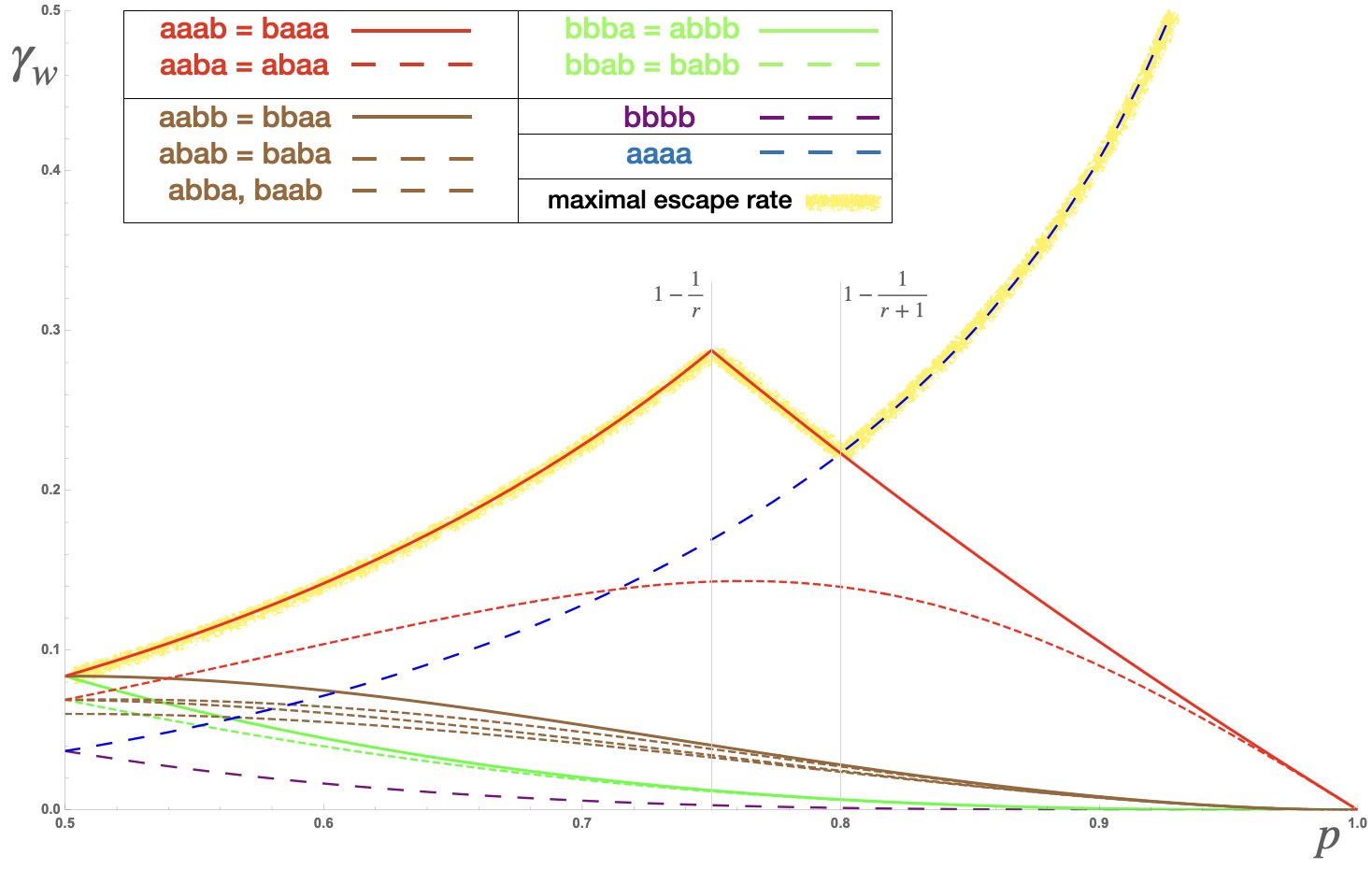}
\caption{The escape rates $\escrate{w}$ for $p\in [\frac 12,1]$, for all holes of length $r=4$.} \label{z0-four-simb}
\end{center}
\end{figure}
\end{example}

\begin{itemize}
\item[$\bullet \, (r=1)$] It is an elementary computation to show that the word $w=(a)$ has the maximal escape rate as $\escrate{a}=- \log (1-p)=- \log  q$, and $\escrate{b}=- \log p$.
\vspace{0.2cm}
\item[$\bullet \, (r=2)$] We have two families:
\begin{itemize}
\item[(i)] $w \in \{(aa), (bb)\}$. Let us start from the case $w=(aa)$. The autocorrelation vector is $c=(1,1)$, and $N_{aa}(a,1,2) =1$, $N_{aa}(b,1,2)=0$, hence the weighted autocorrelation polynomial is
\[
c_{aa}(x,y,z) = 1+ xz.
\]
Moreover $N_{aa}(a)=2$ and $N_{aa}(b)=0$, hence
\[
\tau_{aa}(z) = p^2 z^2 + (1-z) \, c_{aa}(p,q,z) = 1-qz -pq z^2,
\]
from which $z_0 = \frac{-q + \sqrt{q^2+4pq}}{2pq}$ and $\escrate{aa}= \log (\frac{-q+\sqrt{q^2+4pq}}{2pq})$.

The case $w=(bb)$ works as the previous one, we simply need to interchange the roles of $p$ and $q$. Hence $\escrate{bb}= \log (\frac{-p+\sqrt{p^2+4pq}}{2pq})$.

\item[(ii)] $w \in \{(ab), (ba)\}$ Let us start from the case $w=(ab)$.  The autocorrelation vector  is $c=(1,0)$, hence the weighted autocorrelation polynomial is
\[
c_{ab}(x,y,z) = 1.
\]
Moreover $N_{ab}(a)=1$ and $N_{ab}(b)=1$, hence
\[
\tau_{ab}(z) = pq z^2 + (1-z) \, c_{ab}(p,q,z) =(1-pz)(1-qz),
\]
from which $z_0 = \frac 1p$ since $p\ge q$. Hence $\escrate{ab}= - \log p$.

For $w=(ba)$ it turns out that $\tau_{ba}(z)= \tau_{ab}(z)$, hence $\escrate{ba}= \escrate{ab}$. \vspace{0.1cm}
\end{itemize}

Note that for $\frac 12 \le p<\frac 23$, the holes with maximal escape rate are $(ab)$ and $(ba)$, whereas for $p>\frac 23$, the hole with maximal escape rate is $(aa)$ . At the same time, for all values of $p\in (\frac 12,1]$, the hole $(aa)$ has measure $\m(aa)=p^2$, greater than the measure $\m(ab)=\m(ba)=pq$.
\vspace{0.2cm}
\item[$\bullet \, (r=3)$] This case is similar to the  case $r=4$ discussed below. Some further details can be found in \cite{CDK}.
\vspace{0.2cm}
\item[$\bullet \, (r=4)$]
Instead of giving explicit formulas (no more difficulties arise in their derivation than in the case  $r=2$, but the formulas are longer and not very informative), it is more illustrative to plot the different escape rates (see Fig.~\ref{z0-four-simb}) and discuss some important features. Firstly, note that for $p=\frac 12$ all holes have the same measure and their escape rates can be ordered by looking at the length of the shortest periodic pattern they contain, just as discussed in \cite{BY}. On the other hand, it is apparent by Fig.~\ref{z0-four-simb} that this ordering is destroyed as soon as the two symbols are not equiprobable. Nevertheless some new patterns can be derived.

In the figure we have used the following notation: curves with the same color correspond to holes with the same measure; solid curves correspond to prime holes and dashed curves to non-prime ones (the dashed blue and purple curves have different styles as they correspond to holes for which there is no prime hole with the same measure). First, it is  apparent that in the family of holes with the same measure, prime holes (if they exist) are the leakiest (see Lemma \ref{lemma-elem-1}-(i)). For example, the red curves are the plot of the escape rates of the holes of measure $p^3(1-p)$. Among them, the two prime holes $(aaab)$ and $(baaa)$ have (the same) maximal escape rate (look at the solid red curve). On the other hand, the escape rates of prime holes with different measures are ordered, for a fixed $p$, by their measure (see Lemma \ref{lemma-elem-1}-(ii)). One is then tempted to say that the maximal escape rate among all the holes of fixed length should be achieved by the prime hole with maximal measure (red solid curve). But this is not true for all values of $p$, as  is apparent by looking at Fig.~\ref{z0-four-simb}: for $p> 1-\frac{1}{r+1}=\frac 45$ (see Theorem \ref{thm:biggest-two-symb}) the maximal escape rate is given by the hole $(aaaa)$, i.e., the hole with the repetition of the most probable symbol. Finally, the escape rate of the  hole $(aaaa)$ is always greater than that of the similar hole $(bbbb)$ (see Lemma \ref{lemma-elem-1}-(iii)).
\end{itemize}
The formal derivation of the hole with maximal escape rate is our most important result. We give it in the next section.

\subsection{Main results}
Let the alphabet $\AAA$ with $A\ge 2$ symbols be fixed and  denote by $a$ the most probable symbol  and by $b$  the second most probable one. Hence we rename by $\{a,b,a_3,\dots,a_A\}$ the symbols in $\AAA$, with $p= p_a$, $q=p_b$; $q+p\le 1$  and $p_{a_k}\le q \le p$ for all $k\ge 3$.

As anticipated, an important role is played by prime holes and by holes which have maximal measure for a fixed length.
Using the expression of the polynomial $\tau_{w}$ in \eqref{pol-tau}, we obtain a few basic  inequalities which arise when comparing their escape rates:

\begin{lemma}\label{lemma-elem-1} Let us consider a fixed $r\ge 2$. Then:
\begin{enumerate}[(i)]
\item Let $w_1$ and $w_2$ be holes of the same length and measure. If $w_1$ is prime then $\escrate{w_1} \ge \escrate{w_2}$, and  equality holds only if $w_2$ is also prime.
\item Let $w_1$ and $w_2$ be prime holes of the same length. If $\mw{1}> \mw{2}$ then $\escrate{w_1}> \escrate{w_2}$.
\item Let $w_1=(a_i a_i\dots a_i)$ and $w_2=(a_j a_j\dots a_j)$ be holes of the same length given by repetitions of different symbols. If $p_{a_i}\ge p_{a_j}$ then $\escrate{w_1}\ge \escrate{w_2}$.
\end{enumerate}
\end{lemma}

\begin{proof} (i) If $w_1$ is prime and $w_2$ is not prime, then $c_{w_1}(p_a,p_b,\dots,p_{a_A},z)=1 < c_{w_2}(p_a,p_b,\dots,p_{a_A},z)$ for all $z>0$. Hence by \eqref{pol-tau}, for $z>1$ we have $\tau_{w_1} (z) > \tau_{w_2} (z)$ and therefore $\escrate{w_1} > \escrate{w_2}$. If both are prime then $\tau_{w_1} (z) = \tau_{w_2} (z)$.\\
(ii) Note that for a prime hole $w$ we have $c_w(p_a,p_b,\dots,p_{a_A},z)=1$ and, as shown in \eqref{pol-prime}, the polynomial $\tau_w(z)$ belongs to the family studied in Lemma \ref{lemma-tau}. Hence, if $w_1$ and $w_2$ are prime holes of the same length, it follows that $\tau_{w_1} (z) = f_{m_1}(z)$ with $m_1=\mw{1}$, and $\tau_{w_2} (z) = f_{m_2}(z)$ with $m_2=\mw{2}$. Hence if $\mw{1}> \mw{2}$ then $\tau_{w_1} (z) > \tau_{w_2} (z)$ for all $z>0$ and $\escrate{w_1}> \escrate{w_2}$.\\
(iii) Let us assume for simplicity that $a_i = a_1 = a$ and $a_j = a_2 = b$, so that $p = p_{a_i}$ and $q = p_{a_j}$, and note that the proof in no way depends on these being the most probable symbols. Use \eqref{pol-constant} to write
\[
\tau_{w_1}(z) = \frac{p^r(1-p)\, z^{r+1}-z+1}{1-p z},\quad \tau_{w_2}(z) = \frac{q^r(1-q)\, z^{r+1}-z+1}{1-q z}.
\]
As remarked before, both the numerator and the denominator of $\tau_{w_1}(z)$ vanish at $z=p^{-1}$, and the same is true for  $\tau_{w_2}(z)$ at $z=q^{-1}$. In addition, by Lemma \ref{lemma-tau}-(iii), the numerator of $\tau_{w_1}(z)$ has two distinct positive roots for $p\not= 1-\frac{1}{r+1}$, and one double positive root for $p= 1-\frac{1}{r+1}$ . This follows by using the function $[0,1] \ni x\mapsto g(x):= x^{r}(1-x)$, which has a strict maximum at $1-\frac{1}{r+1}$. Then
\[
g(p) = p^{r}(1-p) \le g\left(1-\frac{1}{r+1}\right) = \frac1{r+1} \left(1-\frac1{r+1}\right)^{r} = m_{r+1}^*\, , \quad \forall\, p\in [0,1],
\]
and $p^{r}(1-p) = m_{r+1}^*$ if and only if $p=1-\frac1{r+1}$. The analogous result holds for the numerator of $\tau_{w_2}(z)$. Hence $\tau_{w_1}(z)$ and $\tau_{w_2}(z)$ have at most one positive root each, which we denote, respectively, $z_0(\tau_{w_1})$ and $z_0(\tau_{w_2})$.

We now claim that $p^r(1-p) \ge q^r(1-q)$. The function $g(x)= x^r(1-x)$ is increasing in $[0,\frac{r}{r+1}]$ and decreasing in $[\frac{r}{r+1},1]$. Since $p \ge q$ and $p+q\le 1$, if $p\le \frac{r}{r+1}$ then $g(p)\ge g(q)$ and the claim follows. If $p\ge \frac{r}{r+1}$ then $p > \frac 12$, so that $(1-p) < p$ and $q\le (1-p) < \frac 12 < \frac{r}{r+1}$. Then $g(q)\le g(1-p)$, hence $q^r(1-q) \le (1-p)^r p \le p^r (1-p)$, since $(1-p)^{r-1}\le p^{r-1}$. The claim is proved.

Finally, the claim implies that the numerator of $\tau_{w_1}(z)$ is greater or equal than that of $\tau_{w_2}(z)$ for all $z\ge 0$, hence $z_0(\tau_{w_1})\ge z_0(\tau_{w_2})$, and thus $\escrate{w_1}\ge \escrate{w_2}$.
\end{proof}

We now use  Lemma \ref{lemma-elem-1} to  compare all the different holes with the same length and conclude that the hole with maximal escape rate is either a prime hole or a hole with maximal measure. To this end, it is helpful
 to define the following families of holes:
\begin{align}
& P^r := \set{w \in \AAA^r : \text{$w$ is a prime hole, $\m(w)\ge \m(\tilde w)$ for all prime holes $\tilde w \in \AAA^r$}}, \label{hole-prime-maximal}\\
& M^r := \set{w \in \AAA^r : \text{$\m(w)\ge \m(\tilde w)$ for all $\tilde w \in \AAA^r$}}. \label{hole-maximal}
\end{align}
In words, $P^r$ is the set of words of length $r$ for which the corresponding hole is prime and has maximal measure among all the prime holes; $M^r$ consists of the holes of maximal measure. In what follows, with an oversimplification of notation, when we write $\primew$ and $\maxw$ we will implicitly intend that $\primew \in P^r$ and $\maxw \in M^r$.

Using the definition of the symbols $a$ and $b$, one immediately verifies that
\[
\set{(aaa\dots ab)\, ,\, (ba\dots aaa)} \subseteq P^r,
\]
thus $\m(\primew) = p^{r-1}q$ for all $\primew \in P^r$.
On the other hand, since $a$ is the most probable symbol, we have $(aaa\dots a)\in M^r$ and $\m(\maxw)=p^r$. Notice that the hole $(aaa\dots a)$ is not prime.

\begin{rem} \label{rem-imp}
For what follows it is useful to know when the two sets $P^r$ and $M^r$ are disjoint or not.

Let $q<p$. In this case $M^r = \{ (aaa\dots a)\}$, so that $M^r\cap P^r = \emptyset$. Moreover, for all prime holes $w$, $\m(w) < \m(aaa\dots a)$. On the contrary, for many $w\not = (aaa\dots a)$ there exists a prime hole $\tilde w$ such that $\m(w) = \m(\tilde w)$. In fact, whenever a word contains at least two different symbols, there exists a prime hole with its same measure: it is indeed easy to produce a prime hole which contains any symbol of $\AAA$ any number of times, for example
\[
w=(a_1\dots a_1 a_2\dots a_2 \dots a_A\dots a_A).
\]
Observe however that this argument does not work for a word given by $r$ repetitions of a single symbol with probability different from those of all other symbols.

If $q=p$ then $P^r \cap M^r \neq \emptyset$ as the words $(aaa\dots ab)$ and $(ba\dots aaa)$ are in the intersection of the two sets.
\end{rem}

We can now state our first main result:
\begin{theorem}\label{thm-elem-2}
Let $r\ge 2$ be a fixed word length. The escape rate $\escrate{\primew}$ is the same for all $\primew \in P^r$, cf.\ \eqref{hole-prime-maximal}, and the escape rate $\escrate{\maxw}$ is the same for all $\maxw \in M^r$, cf.\ \eqref{hole-maximal}. Moreover
\[
\gamma_{max}^r:= \max \set{ \escrate{w}\, :\, w \text{ has length } r} = \max \{ \escrate{\primew}\, ,\, \escrate{\maxw} \}.
\]
\end{theorem}

\begin{proof}
The first assertion comes from Proposition \ref{prop:escape-rate} and \eqref{pol-prime}-\eqref{pol-constant}. For a fixed length $r\ge 2$, we can first group the holes $w$ according to their measures. Then by Lemma \ref{lemma-elem-1}-(i),  for fixed length and measure, the hole  with maximal escape rate is prime, whenever a prime hole with that given measure exists. By Remark \ref{rem-imp}, if $q<p$ there is no prime hole of measure $p^r$ since $P^r\cap M^r = \emptyset$, and there is no prime hole of measure $p_k^r$, if $p_k$ is the probability of a symbol in $\AAA$ and no other symbol has the same probability. If $q=p$ instead, there is a prime hole of measure $p^r$.

From the above arguments, it remains only to consider the set of prime holes and the set of holes which have measure different from that of all the prime holes. As explained in Remark \ref{rem-imp}, these last cases correspond to holes of the form $(a_i a_i \dots a_i)$, words with one single symbol repeated $r$ times. Applying now Lemma \ref{lemma-elem-1}-(ii), the hole with the maximal escape rate among the prime holes is $\primew$, and applying Lemma \ref{lemma-elem-1}-(iii), the hole with the maximal escape rate among the holes with one single symbol repeated is $\maxw$.

Finally, to obtain the maximal escape rate it is sufficient to compare $\escrate{\primew}$ and $\escrate{\maxw}$.
\end{proof}

\subsubsection{Explicit expression for the maximal escape in the case of two symbols}
We now show that if $A=2$ we can explicitly identify, for all $r\ge 2$, a hole with maximal escape rate.

\begin{theorem}\label{thm:biggest-two-symb}
Let $\AAA=\{a,b\}$ with $p=p_a\ge q=p_b$ satisfying $p+q=1$. For holes $w$ of fixed length $r\ge 2$,
\[
\gamma_{max}^r= \left\{ \begin{array}{ll} \escrate{\primew} \, , & \text{if } p\in \left[\frac 12 \, ,\,  1-\frac{1}{r+1}\right]; \\[0.3cm] \escrate{\maxw} \, , &\text{if } p\in \left[1-\frac{1}{r+1}\, ,\, 1\right); \end{array} \right.
\]
where $\primew$ denotes a word in $P^r$, cf.\ \eqref{hole-prime-maximal}, and $\maxw$ denotes a word in $M^r$, cf.\ \eqref{hole-maximal}. In addition, for $p\in [1-\frac 1r\, ,\, 1-\frac{1}{r+1}]$ we can explicitly compute that $\escrate{\primew} = \log \frac 1p$ and thus obtain that in this range $\gamma_{max}^r = \log \frac 1p$.
\end{theorem}

\begin{proof}
It is enough to show that it is possible to determine which of the escape rates of the holes $\primew$ and $\maxw$ is maximal, and apply Theorem \ref{thm-elem-2}.

The case $r=2$ is studied in details in Example \ref{two-symbols}. Let us consider a fixed length $r\ge 3$ and start with the case $p>q$, hence $p>\frac 12$. We first deal with the hole $\primew$, which is a prime hole with maximal measure among the prime holes. An example is the hole $(aaa\dots ab)$. We have $\m(\primew)=p^{r-1}q$ and $\tau_{\primew}(z)$ is a polynomial of the family studied in Lemma \ref{lemma-tau} given by
\[
\tau_{\primew}(z) = p^{r-1}q \, z^r -z+1.
\]
Applying Lemma \ref{lemma-tau} to $\tau_{\primew}(z)$ with $m=p^{r-1}q = p^{r-1}(1-p)$, it follows that $\tau_{\primew}(z)$ has two distinct positive roots for $p\not= 1-\frac 1r$ and one double positive root for $p=1-\frac 1r$. To show this it is enough to repeat the argument in the proof of Lemma \ref{lemma-elem-1}-(iii), using now the function $[0,1] \ni x\mapsto g(x):= x^{r-1}(1-x)$. Let us denote by $z_0(\primew)$ and by $z_1(\primew)$ the smallest and the biggest positive roots of $\tau_{\primew}(z)$ respectively. Since $z_r^*(p^{r-1}q)$ is the point of minimum for the polynomial $\tau_{\primew}$ we have
\[
z_0(\primew) \le z_r^*(p^{r-1}q) = \frac 1p \,(r\, q)^{-\frac{1}{r-1}}\le z_1(\primew).
\]
Since $\tau_{\primew}(\frac 1p)=0$ for all $p$, either $z_0(\primew)$ or $z_1(\primew)$ is $\frac 1p$. It is clear from the previous estimate that if $(r\, q)^{-\frac{1}{r-1}} <1$ then $z_1(\primew)=\frac 1p$, and if $(r\, q)^{-\frac{1}{r-1}} \ge 1$ then $z_0(\primew)=\frac 1p$. Thus, using $q=1-p$, we have proved that
\begin{equation}\label{er-h*}
\escrate{\primew} \left\{ \begin{array}{ll} = \log \frac 1p\, , & \text{if } p\in \left[ 1-\frac 1r\, ,\, 1\right);\\[0.3cm] < \log z_r^*(p^{r-1}q)\, , & \text{if } p\in \left[\frac 12 \, ,\, 1-\frac 1r\right). \end{array} \right.
\end{equation}

Let us now consider the hole $\maxw$, which is a hole with maximal measure among all the holes of length $r$. An example is the hole $\maxw=(aaa\dots a)$, and this is the only example for $p>q$, that is, for $p>\frac 12$. Using \eqref{pol-constant}, we have
\[
\tau_{\maxw}(z)= \frac{p^r(1-p)\, z^{r+1}-z+1}{1-pz}.
\]
In particular we know that the numerator has two distinct positive roots for $p\not= 1-\frac{1}{r+1}$ and one of the positive roots is always $\frac 1p$. Since $\frac 1p$ is also the unique root of the denominator, it follows that $\tau_{\maxw}(z)$ has only one positive root $z_0(\maxw)$, which is the positive root of the numerator not equal to $\frac 1p$. Hence
\[
\escrate{\maxw} = \log z_0(\maxw).
\]

From the previous argument on prime holes, we also know that $z_0(\maxw)>\frac 1p$ if and only if $p\ge 1-\frac{1}{r+1}$, hence using \eqref{er-h*} we conclude that $\gamma_{max}^r = \escrate{\maxw}$ for $p\ge 1-\frac{1}{r+1}$, and $\gamma_{max}^r = \escrate{\primew}=\log \frac 1p$ for $p\in [1-\frac 1r\, ,\, 1-\frac{1}{r+1}]$.

To conclude the argument, we need to consider the case $p\in (\frac 12,1-\frac 1r]$, for which both $z_0(\primew)$ and $z_0(\maxw)$ are smaller than $\frac 1p$. To compare these two values, we introduce the following notation. Let
\begin{equation} \label{pol-speciale}
\bar \tau_{r,p}(z) := p^{r-1}(1-p)\, z^r -z+1,
\end{equation}
then
\[
\tau_{\primew}(z) = \bar \tau_{r,p}(z)\quad \text{and} \quad \tau_{\maxw}(z)= \frac{ \bar \tau_{r+1,p}(z)}{1-pz}.
\]
By the previous arguments, the polynomial $\bar \tau_{r,p}(z)$ has two positive roots, one is $\frac 1p$ and let us denote the other by $\bar z_r(p)$. Recall that $\bar z_r(p) = \frac 1p$ if and only if $p=1-\frac 1r$. In addition, we have shown that for $p\in (\frac 12,1-\frac 1r]$, one has $z_0(\maxw) = \bar z_{r+1}(p)< \frac 1p$ and $z_0(\primew)= \bar z_r(p)<\frac 1p$. Since $\bar \tau_{r,p}(z) \ge \bar \tau_{r+1,p}(z)$ for $z\in [1,\frac 1p]$ and for all $r$, it follows that $\bar z_r(p) \ge \bar z_{r+1}(p)$ for $p\in (\frac 12,1-\frac 1r]$. Hence $\gamma_{max}^r = \escrate{\primew}=\log \bar z_r(p)$ for $p\in (\frac 12\, ,\, 1-\frac 1r]$.

Finally, if $r\ge 3$ and $p=q=\frac 12$, we know that $(aa\dots ab)$ is an example of a word in $P^r\cap M^r$, and by Theorem \ref{thm-elem-2}, it follows $\gamma_{max}^r = \escrate{\primew}=\escrate{\maxw}$.

We have thus finished the proof of the theorem, and can collect all the information on the maximal escape rate by saying that
\[
\gamma_{max}^r= \left\{ \begin{array}{ll} \escrate{\primew} = \log \bar z_r(p) \, , & \text{if } p\in \left[\frac 12 \, ,\,  1-\frac{1}{r}\right]; \\[0.3cm] \escrate{\primew} = \log \frac 1p\, , & \text{if } p\in \left[1-\frac 1r\, ,\, 1-\frac{1}{r+1}\right];\\[0.3cm] \escrate{\maxw} = \log \bar z_{r+1}(p) \, , &\text{if } p\in \left[1-\frac{1}{r+1}\, ,\, 1\right). \end{array} \right.
\]\end{proof}

\subsubsection{Maximal escape in the case of more than two symbols}
The situation is more intricate in the case with more than two symbols, that is $A>2$, as elucidated by the following:

\begin{proposition}\label{more-symbols}
With the notation of Theorem \ref{thm-elem-2}, let the alphabet $\AAA$ have $A>2$ elements, and let $a$ and $b$ be the two most probable symbols with probabilities given by $p$ and $q$ respectively. Let $r\ge 2$ be a fixed length, then $\gamma_{max}^r = \escrate{\maxw}$ for $p\ge 1-\frac{1}{r+1}$. In addition, if $q<p(1-p)$ then $\gamma_{max}^r = \escrate{\maxw}$ for $p\in [\frac 12, 1-\frac{1}{r+1}]$. On the other hand, there exist values of $q$ sufficiently close to $1-p$ and of $p\in (\frac 12, 1-\frac{1}{r+1})$, such that $\gamma_{max}^r = \escrate{\primew}$.
\end{proposition}

\begin{proof}
Let us first consider the hole $\maxw =(aaa\dots a)$. As in the proof of Theorem \ref{thm:biggest-two-symb} we can write
\[
\tau_{\maxw}(z) = \frac{\bar \tau_{r+1,p}(z)}{1-pz}
\]
where $\bar \tau_{r+1,p}(z)$ is as in \eqref{pol-speciale}, and $\escrate{\maxw} = \log \bar z_{r+1}(p)$, whence $\escrate{\maxw} \ge \log \frac 1p$ for $p\ge 1-\frac{1}{r+1}$.

On the other hand, a prime hole $\primew$ of length $r$ has measure $p^{r-1}q$ and polynomial
\[
\tau_{\primew}(z) = p^{r-1}q\, z^r-z+1 \le \bar \tau_{r,p}(z),\quad \forall\, z>0 \,,
\]
because $q\le 1-p$. Hence by the proof of Theorem \ref{thm:biggest-two-symb}, $\escrate{\primew}$ is smaller than $\log \frac 1p$ for $p\ge 1-\frac 1r$.

We have thus proved, using Theorem \ref{thm-elem-2}, that $\gamma_{max}^r = \escrate{\maxw}$ for $p\ge 1-\frac{1}{r+1}$.

Let us now assume that $q<p(1-p)$. Then  $\tau_{\primew}(z)< \bar \tau_{r+1,p}(z)$ for all $z>1$, and since $z_0(\primew)>1$ it follows that $z_0(\primew) < \bar z_{r+1}(p)$ for all $p$. Then $\escrate{\primew} \le \escrate{\maxw}$ for all $p\in [\frac 12, 1]$.

Note that the previous result is different from that for shifts on two symbols. However, all the quantities that we are using have continuous dependence on the probabilities of the symbols. Hence, if $q$ is sufficiently close to $1-p$, that is if we are sufficiently close to the case of shifts on two symbols, we expect to find the same kind of results obtained in Theorem \ref{thm:biggest-two-symb}. Therefore there exist values of $p< 1-\frac{1}{r+1}$ for which the maximal escape rate is achieved by a prime hole $\primew$.
\end{proof}

\begin{rem}\label{rem-bunim}
In the case of two equiprobable symbols, $A=2$ with $p=1-p=\frac 12$, in \cite{BY} the authors prove that prime holes have the maximal escape rate among the holes with the same measure and length, but also show that it is possible to order same-measure holes according to their escape rate by using the minimal period of periodic points in the hole. In this paper we have proved that prime holes have maximal escape rate among the holes with the same measure and same length, also in the case of non-equiprobable symbols, that is $A=2$ and $p\not= \frac 12$. One may wonder whether also the ordering found in \cite{BY} for non-prime holes is preserved  when the symbols are not equiprobable.  We show that this is not the case. We find same-length words $w$ and $\tilde w$ such that the corresponding holes are not prime and have the same measure, and such that there exists $p^* \in (\frac12, 1)$ with $\escrate{w} > \escrate{\tilde w}$, for $p\in (\frac 12, p^*)$, and $\escrate{w} < \escrate{\tilde w}$ for $p\in (p^*,1)$. Hence the ordering of the holes does not only depend on the length of the periodic orbits in the hole. One can check that this phenomenon occurs for example for $w=(aabbaa)$ and $\tilde w=(baaaab)$, with $p^* \approx \frac{\sqrt{2}}{2}$: the hole $w$ contains a periodic orbit with period four, whereas the minimal period of the periodic orbits contained in the hole $\tilde w$ is five. For holes shrinking to a periodic point the escape rate behaves as the instability factor of the orbit. In this case we remark that the two factors for $w$ and $\tilde w$ are the same for $p=\frac{\sqrt{5}-1}{2}< p^*$, hence this does not seem to be the reason for the order switching found above.
\end{rem}

\subsubsection{Estimates}

By Theorem \ref{thm-elem-2}, the maximal escape rate may be obtained simply by comparing the roots of the polynomials $\tau_{w_P}$ and $\tau_{w_M}$. While for small $r$ such roots can be computed exactly, for large $r$ one should rely on numerical approximations, that in principle provide a value with arbitrary precision. On the other hand, if a numerical approximation is not at hand, it could be relevant to have rigorous estimates of the maximal escape rate, in particular for given $p$ and large length $r$. In this section we show that, with elementary arguments, one can obtain explicit estimates of the maximal escape rate in the case of a two-symbol alphabet (see Fig.~\ref{FigRelErr} for examples).
\vspace{0.3cm}

We start by giving an estimate on the escape rate for prime holes which holds for all finite alphabets, $A\ge 2$.

\begin{lemma}\label{lemma-elem-new} Let us consider a fixed $r\ge 2$. If $w$ is a prime hole of length $r$ we have
\[
\log\left( \frac{1+r(r-2)\m(w)-\sqrt{1-r\m(w)(2+(r-2)\m(w))}}{r(r-1)\m(w)} \right) \le \escrate{w} \le \frac{1}{r-1} \log \frac{1}{r \m(w)}.
\]
\end{lemma}

\begin{proof} For a prime hole $w$ we have $\m(w) \le \m(\primew) = p^{r-1} q \le p^{r-1}(1-p)$ where $p$ and $q$ are the probabilities of the most probable and of the second most probable symbols. As in the proof of Lemma \ref{lemma-elem-1}-(iii), using the function $g(x):= x^{r-1}(1-x)$ on $[0,1]$, one has
\[
p^{r-1}(1-p) \le m_r^*= \frac 1r \left(1-\frac 1r\right)^{r-1}, \quad \forall\, p\in [0,1],
\]
and $p^{r-1}(1-p) = m_r^*$ if and only if $p=1-\frac 1r$. Hence $\m(w)\le m_r^*$ for all prime holes $w$, and $\m(w)= m_r^*$ if and only if $\m(w)= p^{r-1}(1-p)$ for $p=1-\frac 1r$. Since we can apply Lemma \ref{lemma-tau} to $\tau_w(z)$, we obtain that $\tau_w(z)$ has at least one root $z_0>1$, and if $z_0$ is the smallest positive root then
\[
z_0 \le z^*_r(\m(w)) =  \left( \frac{1}{r\, \m(w)}\right)^{\frac{1}{r-1}}.
\]
To prove the bound from below, recall that by Lemma \ref{lemma-tau} the polynomial $\tau_w(z)$ is convex in $(0,+\infty)$, and for the derivatives we have $\tau^{(j)}_w(1)>0$ for all $j\ge 2$. Hence $z_0$, the smallest positive root of $\tau_w(z)$, is greater than 1, and we can bound $\tau_w(z)$ from below by its osculating parabola at $z=1$, that is
\[
\tau_w(z) \ge \frac 12\, \tau''_w(1) (z-1)^2 + \tau'_w(1)\, (z-1) + \tau_w(1), \quad \forall\, z\ge 1.
\]
Since $\tau_w(1)=\m(w)>0$ and $\tau'_w(1)=r\m(w)-1<0$, the roots of the osculating parabola are both greater than 1. It follows that the smallest positive root $z_0$ of $\tau_w(z)$ is greater than the smallest root of the osculating parabola.
\end{proof}

When the alphabet has two symbols, by Theorem \ref{thm:biggest-two-symb} we know which hole has the maximal escape rate, and it is simpler to estimate the smallest positive root of the associated polynomial also thanks to Lemma \ref{lemma-elem-new}.

\begin{corollary} \label{cor:stime}
Let $\AAA=\{a,b\}$ with $a$ the symbol with largest probability $p\in [\frac 12,1]$. For holes $w$ of fixed length $r\ge 2$ the maximal escape rate satisfies:
\begin{itemize}
\item[(i)] If $p\in\left[ \frac 12, 1- \frac1r\right)$ then
\[
\gamma_{max}^r \in \left[\log \bar\gamma(p,r)\, ,\, \frac{1}{r-1}\, \log \frac{1}{rp^{r-1}(1-p)}\right]
\]
where
\[
\bar \gamma(p,r) = \frac{1+r(r-2)\m(w_P)-\sqrt{1-r\m(w_P)(2+(r-2)\m(w_P))}}{r(r-1)\m(w_P)}
\]
with $\m(w_P) = p^{r-1}(1-p)$.
\item[(ii)] If $p\in\left[ 1- \frac1r, 1-\frac{1}{r+1}\right]$ then
\[
\gamma_{max}^r = \log \frac 1p.
\]
\item[(iii)] If $p\in\left(1-\frac{1}{r+1},1\right)$ then
\[
\gamma_{max}^r \in \left[\log \frac 1p + \frac 1r \, \log \frac{1}{(r+1)(1-p)}\, ,\, \log \left(\frac{-1+p + \sqrt{(1-p)^2+4p(1-p)}}{2p(1-p)}\right) \right].
\]
\end{itemize}
\end{corollary}

\begin{proof}
For $p\le 1-\frac{1}{r+1}$ we apply Theorem \ref{thm:biggest-two-symb} and Lemma \ref{lemma-elem-new} to the hole $w_P$ for which $\m(w_P)=p^{r-1}(1-p)$.

For $p> 1-\frac{1}{r+1}$, we use that $\gamma_{max}^r = \escrate{w_M}$ and that
\[
\tau_{\maxw}(z)= \frac{ \bar \tau_{r+1,p}(z)}{1-pz}
\]
with $\bar{\tau}_{r+1,p}$ defined in \eqref{pol-speciale}. It follows that $\escrate{w_M} = \log z_0(w_M)$, where $z_0(w_M) > \frac 1p$ is one of the two positive roots of the polynomial $\bar{\tau}_{r+1,p}$ defined in \eqref{pol-speciale}, the other being $\frac 1p$. By Lemma \ref{lemma-tau}, it follows that $z_0(w_M)$ is the largest root of $\bar{\tau}_{r+1,p}$, and then with $q=1-p$,
\[
z_0(w_M) \ge z_{r+1}^*(p^r q) = \frac 1p\, \Big( q(r+1)\Big)^{-\frac 1r}.
\]
This gives the lower bound for $\gamma_{max}^r$ for $p> 1-\frac{1}{r+1}$. To obtain the upper bound, we use \eqref{pol-constant} to write
\[
\tau_{\maxw}(z) = p^rz^r + (1-z)\, \sum_{j=0}^{r-1}\, p^j z^j = 1- (1-p) \sum_{j=1}^r\, p^{j-1}\, z^j.
\]
It follows that $\tau_{\maxw}(z)$ is less than any truncated sum of its terms for all $z\ge 0$. In particular, truncating the sum at $k=2$ one gets
\[
\tau_{\maxw}(z) \le 1-(1-p)z -p(1-p)z^2 = \tau_{aa}(z)
\]
for all $z\ge 0$, where $\tau_{aa}(z)$ is the polynomial associated to the hole $(aa)$ in Example \ref{two-symbols}. In particular the unique positive root of $\tau_{\maxw}(z)$ is less or equal than the unique positive root of $\tau_{aa}(z)$.
\end{proof}

 Finally, we remark that the lower bounds turn out to be quite precise: the relative error between the exact value and the estimates decays to zero exponentially fast with the length $r$ (see Fig.~\ref{FigRelErr}).

\begin{figure}[ht!]
\centerline{\includegraphics[width=8.5cm,keepaspectratio]{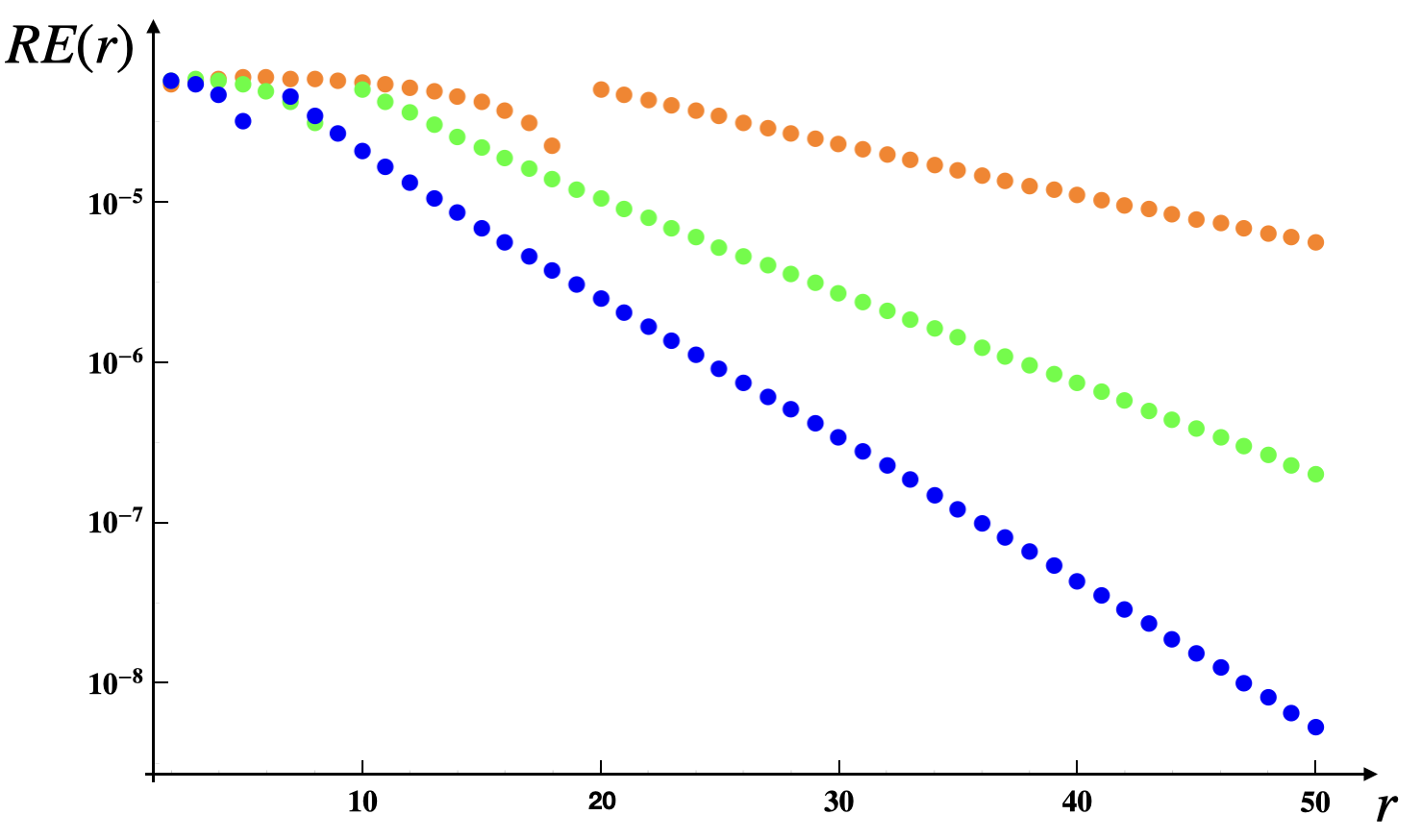}}
\caption{The relative error between a very precise numerical approximation of $\gamma^r_{max}$ and the lower bound $l_b$ in Corollary \ref{cor:stime}, defined by $RE(r):=(\gamma^r_{max} - l_b)/ \gamma^r_{max}$ and displayed as a function of the length $r$ of the hole in log-linear scale. The decay towards zero shows that the accuracy of the estimate improves exponentially with the length of the hole. Different curves correspond to different values of $p$ (from bottom to top: $p=0.85$, $p=0.9$, $p=0.95$).}
\label{FigRelErr}
\end{figure}

\section{The case of Markov measures} \label{sec:markov}
In the previous sections we have considered the dynamical system $(\AAA^{\N},\sigma)$ endowed with a product probability measure. In this section we discuss the extension of some of our results to the case of \emph{Markov measures} limiting ourselves to the alphabet $\AAA=\{a,b\}$ with two symbols. Given a stochastic matrix
\[
\Pi = \begin{pmatrix} \pi_{aa} & \pi_{ab} \\ \pi_{ba} & \pi_{bb} \end{pmatrix}
\]
with $\pi_{ij}\ge 0$ for all $i,j\in \AAA$ and $\pi_{aa}+\pi_{ab} = \pi_{ba}+\pi_{bb}=1$, we consider the set
\[
\AAA^{\N}_{_\Pi} := \set{ \omega \in \AAA^{\N} \, :\, \pi_{\omega_i \omega_{i+1}} >0\, \text{ for all $i\ge 0$}}
\]
and the action of the shift transformation $\sigma$ on $\AAA^{\N}_{_\Pi}$. One can define analogously the set of allowed finite words $\AAA^{*}_{_\Pi}$. It is well known that if the matrix $\Pi$ is irreducible and aperiodic, that is, there exists $N>0$ such that all the entries of $\Pi^n$ are positive for $n\ge N$, then there is a unique vector $p=(p_a, p_b)$ such that $p_a,p_b >0$, $p_a+p_b=1$ and $p\Pi =p$. In this situation the shift $\sigma$ preserves the probability measure $m_{_\Pi}$, called  the \emph{Markov measure}, defined on finite words $s= (s_0 s_1 \dots s_{k-1}) \in \AAA^k$ to be
\[
m_{_\Pi}(s) = p_{s_0}\, \prod_{j=0}^{k-2}\, \pi_{s_j\, s_{j+1}}.
\]
In this section we are interested to the symbolic dynamical system $(\AAA^{\N}_{_\Pi},m_{_\Pi},\sigma)$, which is well known to be ergodic. We can then study the escape rates for holes in $\AAA^{\N}_{_\Pi}$ given by finite words.

 The following parameter $\chi_{_\Pi} \in (-1,1)$ will play an important role:
\begin{equation} \label{param-chi}
\chi_{_\Pi}:=\pi_{aa}+\pi_{bb}-1.
\end{equation}
Note that for $\chi_{_\Pi}=0$
we have $\pi_{aa}=\pi_{ba}$, hence the rows of $\Pi$ are equal and the Markov measure becomes a product measure. Thus the case $\chi_{_\Pi}=0$ corresponds to those studied in the previous sections.

Special examples of the system $(\AAA^{\N}_{_\Pi},m_{_\Pi},\sigma)$ are \emph{subshifts of finite type}, which correspond to stochastic matrices $\Pi$ with at least one vanishing entry. In addition, shifts with a Markov measure are isomorphic to piecewise linear Markov maps of the interval and to Markov chains, hence our results hold for these classes of systems too. We also mention that, while interesting in their own right, Markov systems are often used as first-order approximations of more general nonlinear systems.

We start by introducing  the autocorrelation polynomial of a finite word adapted to the Markov case:

\begin{definition}\label{Mwap}
Let $s\in \{a,b\}^n$, and let $c=(c_0,\dots,c_{n-1})$ denote its autocorrelation vector given in Definition \ref{wap}. Then the \emph{Markovian weighted autocorrelation polynomial} of $s$ is a polynomial in 5 variables given by
\[
c_{s,M}(y_{aa},y_{ab},y_{ba},y_{bb},z) := \sum_{j=0}^{n-1}\, c_j\, \Big( \prod_{i=1}^j\, y_{s_{n-i-1}\, s_{n-i}} \Big)\, z^j
\]
with the convention $\prod_{i=1}^0\, y_{s_{n-i-1}\, s_{n-i}} =1$.
\end{definition}

In Appendix \ref{app:proof-thm-markov} we prove the following result.

\begin{proposition}\label{prop:escape-rate-markov}
The escape rate $\escrate{w}$ of a hole $w$ of length $r$ is given by
\[
\escrate{w} = \log \, z_0\, ,
\]
where $z_0$ is the smallest positive zero of the polynomial
\begin{equation} \label{pol-tau-markov}
\begin{aligned}
\tau_{w,\Pi}(z) :=\ & \mu_{_\Pi}(w)\, z^r \Big(\pi_{w_{r-1} w_0} -\chi_{_\Pi}\, z\, \delta_{w_0 w_{r-1}}\Big)\\ & + (1- z)\, (1-\chi_{_\Pi}\, z)\, c_{w,M}(\pi_{aa},\pi_{ab},\pi_{ba},\pi_{bb},z).
\end{aligned}
\end{equation}
Here $\mu_{_\Pi}(w) := \prod_{j=0}^{r-2}\, \pi_{w_j w_{j+1}}$ is the $m_{_\Pi}$-measure of the hole $w$ divided by the probability $p_{w_0}$ of the first symbols of $w$,
 the symbol $\delta_{\cdot,\cdot}$ denotes the classical Kronecker delta, and $c_{w,M}(y_{aa},y_{ab},y_{ba},y_{bb},z)$ is the Markovian weighted autocorrelation polynomial of $w$.
\end{proposition}

Note that the polynomial $\tau_{w,\Pi}(z)$ is of degree $r$. In fact the Markovian weighted autocorrelation polynomial can be written as
\[
c_{w,M}(\pi_{aa},\pi_{ab},\pi_{ba},\pi_{bb},z) = \delta_{w_0,w_{r-1}}\, \mu_{_\Pi}(w)\, z^{r-1} + \tilde c_{w,M}(\pi_{aa},\pi_{ab},\pi_{ba},\pi_{bb},z)\, ,
\]
where $\tilde c_{w,M}(\pi_{aa},\pi_{ab},\pi_{ba},\pi_{bb},z) := \sum_{j=0}^{r-2}\, c_j\, \Big( \prod_{i=1}^j\, \pi_{w_{r-i-1}\, w_{r-i}} \Big)\, z^j$. Hence the terms of degree $r+1$ in \eqref{pol-tau-markov} cancel out. Moreover, we can write
\begin{equation}\label{pol-tau-markov-due}
\begin{aligned}
\tau_{w,\Pi}(z) =\ & \tilde \mu_{_\Pi}(w)\, z^r + (1- z)\, (1-\chi_{_\Pi}\, z)\, \tilde c_{w,M}(\pi_{aa},\pi_{ab},\pi_{ba},\pi_{bb},z)\\ & + \delta_{w_0 w_{r-1}} \mu_{_\Pi}(w) (1-(1+\chi_{_\Pi})\,z)\, z^{r-1},
\end{aligned}
\end{equation}
where $\tilde \mu_{_\Pi}(w):=\mu_{_\Pi}(w)\, \pi_{w_{r-1} w_0} = \prod_{j=0}^{r-1}\, \pi_{w_j w_{j+1}}$ with $w_r:= w_0$.

\begin{example}\label{two-symbols-markov}
Consider the case of holes of length $r=2$.
\begin{itemize}
\item[$w \in \{(aa),(bb)\}$.]  Let us start with the case $w=(aa)$. The autocorrelation vector is $c=(1,1)$, and  the Markovian weighted autocorrelation polynomial is
\[
\tilde c_{w,M}(\pi_{aa},\pi_{ab},\pi_{ba},\pi_{bb},z) = 1 \, , \qquad \delta_{w_0,w_{r-1}}\, \mu_{_\Pi}(w)\, z^{r-1} = \pi_{aa}\, z.
\]
Hence with $\tilde \mu_{_\Pi}(w)= \pi^2_{aa}$ we obtain
\[
\tau_{w,\Pi}(z) = -(1-\pi_{aa})(1-\pi_{bb})\, z^2 - \pi_{bb}\, z +1
\]
and the smallest positive zero is given by
\[
z_0 = \left\{ \begin{array}{ll} \frac{-\pi_{bb} + \sqrt{\pi^2_{bb}+4(1-\pi_{aa})(1-\pi_{bb})}}{2(1-\pi_{aa})(1-\pi_{bb})}\, , & \text{if }\, (1-\pi_{aa})(1-\pi_{bb})\not= 0;\\[0.2cm]
\frac{1}{\pi_{bb}}\, , & \text{if }\, (1-\pi_{aa})(1-\pi_{bb})= 0.
\end{array} \right.
\]
\\
The case $w=(bb)$ works analogously, interchanging the role of $\pi_{aa}$ and $\pi_{bb}$. Hence the smallest positive zero of $\tau_{w,\Pi}$ is given by
\[
z_0 = \left\{ \begin{array}{ll} \frac{-\pi_{aa} + \sqrt{\pi^2_{aa}+4(1-\pi_{aa})(1-\pi_{bb})}}{2(1-\pi_{aa})(1-\pi_{bb})}\, , & \text{if }\, (1-\pi_{aa})(1-\pi_{bb})\not= 0;\\[0.2cm]
\frac{1}{\pi_{aa}}\, , & \text{if }\, (1-\pi_{aa})(1-\pi_{bb})= 0.
\end{array} \right.
\]
\item[$w\in \{ (ab),(ba) \} $.] For both holes, the autocorrelation vector is $c=(1,0)$, and  the Markovian weighted autocorrelation polynomial is
\[
c_{w,M}(\pi_{aa},\pi_{ab},\pi_{ba},\pi_{bb},z)= \tilde c_{w,M}(\pi_{aa},\pi_{ab},\pi_{ba},\pi_{bb},z) = 1.
\]
Hence with $\tilde \mu_{_\Pi}(w)= \pi_{ab}\pi_{ba}=(1-\pi_{aa})(1-\pi_{bb})$ we obtain
\[
\tau_{w,\Pi}(z) = \pi_{aa}\pi_{bb}\, z^2 - (\pi_{aa}+\pi_{bb})\, z +1
\]
and the smallest positive zero is given by
\[
z_0 = \frac{1}{\max\{ \pi_{aa},\pi_{bb}\}}.
\]
\end{itemize}
\end{example}

The previous example shows that the identification of the hole with maximal escape rate for shifts with a Markov measure is a much more difficult problem than the system with a product measure. When trying to extend the results in Section \ref{sec:bigger} to this case, one immediately finds differences and subtleties; here we made a first step in this direction. The investigation of the Markov case in its generality is outside the scope of the present work and will be the subject of future study.

We start with the analogue of Lemma \ref{lemma-elem-1}. As we will see, in this case it is useful to compare holes with the same length and quantity $\tilde \mu_{_\Pi}$ introduced in \eqref{pol-tau-markov-due}, which replaces the measure of a hole.

\begin{proposition}\label{prop-chi}
Let $w_1$ and $w_2$ be two holes of the same length with $\tilde \mu_{_\Pi}(w_1)=\tilde \mu_{_\Pi}(w_2)$, and let $w_1$ be prime. If $w_2$ is prime then $\escrate{w_1}= \escrate{w_2}$. If $w_2$ is not prime, then:
\begin{enumerate}[(i)]
\item If $\chi_{_\Pi}>0$ we have $\escrate{w_1}> \escrate{w_2}$.
\item If $\chi_{_\Pi}<0$ and $w_2$ is such that $(w_2)_0 \not= (w_2)_{r-1}$, then $\escrate{w_1}> \escrate{w_2}$.
\end{enumerate}
\end{proposition}

\begin{proof}
Let $w_1$ be a prime hole of length $r$. We have $c_{w_1,M}(z)=1$ and
\[
\tau_{w_1,\Pi}(z) = \tilde \mu_{_\Pi}(w_1)\, z^r +(1-z)(1-\chi_{_\Pi} z).
\]
It is clear that if $w_2$ is prime and $\tilde \mu_{_\Pi}(w_2)=\tilde \mu_{_\Pi}(w_1)$, then $c_{w_2,M}(z)= 1$ and $\tau_{w_2,\Pi} = \tau_{w_1,\Pi}$. Let's assume that $w_2$ is not prime.

(i) Let $\chi_{_\Pi}>0$. Since $\frac{1}{\chi_{_\Pi}}>1$, all the positive roots of $\tau_{w_1,\Pi}(z)$ are contained in the interval $(1,\frac{1}{\chi_{_\Pi}})$. Since $w_2$ is not prime, at least one condition between $\tilde c_{w_2,M}(z)>1$ and $\delta_{(w_2)_0 (w_2)_{r-1}}=1$ holds. Then using \eqref{pol-tau-markov-due}
\[
\tau_{w_2,\Pi}(z) < \tilde \mu_{_\Pi}(w_2)\, z^r +(1-z)(1-\chi_{_\Pi} z) = \tau_{w_1,\Pi}(z), \quad \forall\, z\in \left(1,\frac{1}{\chi_{_\Pi}}\right),
\]
because $(1-z)(1-\chi_{_\Pi} z) <0$ and $(1-(1+\chi_{\Pi})z)<0$ in the interval $(1,\frac{1}{\chi_{_\Pi}})$. Since $\tau_{w_1,\Pi}(z)$ has a root in $(1,\frac{1}{\chi_{_\Pi}})$ it follows that the smallest positive root of $\tau_{w_2,\Pi}$ is smaller than that of $\tau_{w_1,\Pi}$, hence $\escrate{w_1}> \escrate{w_2}$.

(ii) Let $\chi_{_\Pi}<0$. In this case the term $(1-\chi_{_\Pi} z)$ is positive for $z>0$, and we only know that the positive roots of $\tau_{w_1,\Pi}(z)$ are greater than 1. Since $w_2$ is not prime, if $(w_2)_0 \not= (w_2)_{r-1}$ then $\tilde c_{w_2,M}(z)>1$ and using \eqref{pol-tau-markov-due}
\[
\tau_{w_2,\Pi}(z) < \tilde \mu_{_\Pi}(w_2)\, z^r +(1-z)(1-\chi_{_\Pi} z) = \tau_{w_1,\Pi}(z), \quad \forall\, z >1.
\]
Hence  the smallest positive root of $\tau_{w_2,\Pi}$ is smaller than that of $\tau_{w_1,\Pi}$, and $\escrate{w_1}> \escrate{w_2}$.
\end{proof}

When the parameter $\chi_{_\Pi}$ is negative, it is possible to find conditions for a stochastic matrix to have non-prime holes with a larger escape rate than the prime holes with the same $\tilde \mu_{_\Pi}$.

If $w_1=(aabb\dots b)$ and $w_2=(abb\dots ba)$ are two words of length $r\ge 3$ with two symbols $a$ and $r-2$ symbols $b$, then
\[
\tilde \mu_{_\Pi}(w_1)=\tilde \mu_{_\Pi}(w_2)=\pi_{aa}\, \pi_{ab}\, \pi_{ba}\, \pi_{bb}^{r-3}.
\]
The hole $w_1$ is prime whereas $w_2$ is not prime, and
\[
\begin{aligned}
& \tau_{w_1,\Pi}(z) = \tilde \mu_{_\Pi}(w_1)\, z^r + (1-z)(1-\chi_{_\Pi} z), \\
& \tau_{w_2,\Pi}(z) =\tau_{w_1,\Pi}(z) + \pi_{ab}\, \pi_{ba}\, \pi_{bb}^{r-3} (1-(1+\chi_{_\Pi})z)\, z^{r-1}.
\end{aligned}
\]
Therefore $\tau_{w_2,\Pi}(z) > \tau_{w_1,\Pi}(z)$ for $z\in (1,\frac{1}{1+\chi_{_\Pi}})$, and $\tau_{w_2,\Pi}(z) < \tau_{w_1,\Pi}(z)$ for $z>\frac{1}{1+\chi_{_\Pi}}$. Hence $\escrate{w_2}> \escrate{w_1}$ if and only if the smallest positive root of $\tau_{w_1,\Pi}$ is smaller than $\frac{1}{1+\chi_{_\Pi}}$.

In Fig. \ref{markov-r3} we show the behaviour of the escape rates for the holes of length $r=3$ as functions of $\pi_{aa}$ and $\pi_{bb}$. The figure clearly shows that for $\chi_{_\Pi}$ negative there are values of $\pi_{aa}$ and $\pi_{bb}$ for which the maximal escape rate is realized by the holes $(aba)$ and $(bab)$ (green surface), which are neither prime nor a repetition of a single symbol.

\begin{figure}[ht!]
\centerline{\includegraphics[width=9cm,keepaspectratio]{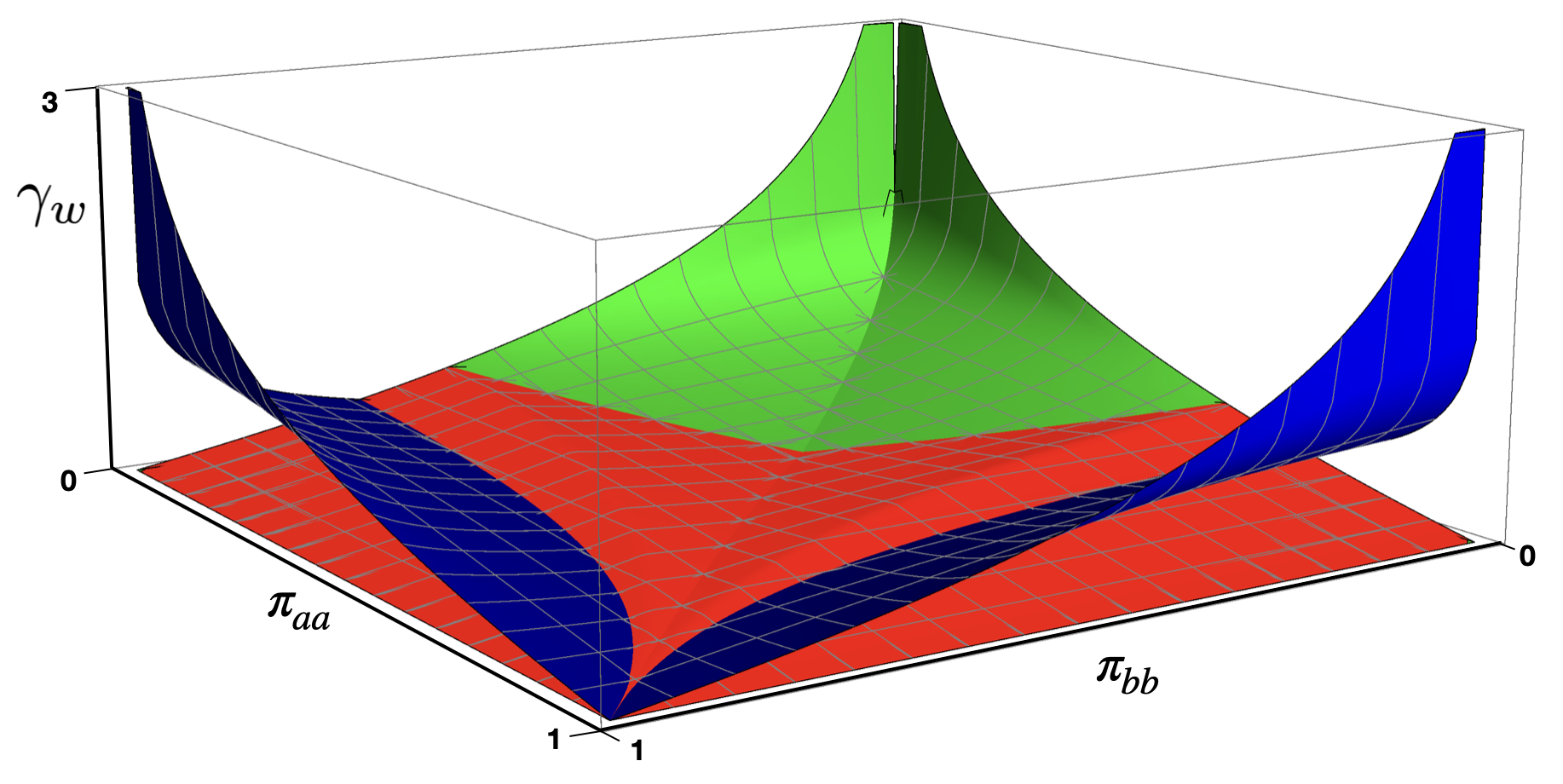}}
\caption{The escape rates for the system of Section \ref{sec:markov}, as functions of
$\pi_{aa},\pi_{bb}\in (0,1)$: the blue graph is for the holes $w=(aaa)$ and $w=(bbb)$;  the red graph is for the holes $w=(aab)$, $w=(bba)$, $w=(baa)$ and $w=(abb)$;  the green graph is for the holes $w=(aba)$ and $w=(bab)$.}
\label{markov-r3}
\end{figure}

\begin{rem}
We briefly compare the results of this section with those of \cite{agar2}. The theorems of \cite{agar2} hold for subshifts of finite type also with more than two symbols (some of the results actually require the number of symbols in the alphabet to be larger than a certain bound) and with respect to the \emph{Parry measure}, the Markov measure of maximal entropy. In the cases covered by both papers, one can easily check that the results obtained by our Proposition \ref{prop:escape-rate-markov} coincide with those in \cite[Thm 2.1]{agar2} (see also \cite[Thm 3.1]{agar}) in the elementary case of holes of length $r=2$, by following the computations in Example \ref{two-symbols-markov} in the case $\pi_{aa}=0$ or $\pi_{bb}=0$ (these are the only cases of a $2\times 2$ irreducible and aperiodic stochastic matrix $\Pi$ with one vanishing entry). For holes of length $r\ge 3$ the escape rates have cumbersome or implicit expressions, therefore the only convenient way to compare escape rate is by numerical approximation.
\end{rem}

\appendix

\section{Proof of Proposition \ref{prop:escape-rate}} \label{app:proof-thm-iid}

Given a hole $w$ of length $r$, we have
\[
S_n = \set{ \omega\in \AAA^{\N}\, :\, \omega^{n+r}\, \text{ does not contain $w$ as a sub-word}}, \quad \forall\, n\ge 0
\]
where $\omega^k\in \AAA^k$ denotes the finite sub-word of $\omega$ given by the first $k$ symbols. Let us introduce the sets
\begin{equation} \label{sigmas}
\Sigma_\ell := \Sigma_\ell^{(w)} := \set{ s \in \AAA^\ell\, :\, s \text{ does not contain $w$ as a sub-word}},
\end{equation}
then clearly $\Sigma_\ell = \AAA^\ell$ for $\ell <r$, and we can write
\begin{equation}\label{esses}
S_n = \set{ \omega\in \AAA^{\N}\, :\, \omega^{n+r}\, \in \Sigma_{n+r}}, \quad \forall\, n\ge 0.
\end{equation}
Let us recall that a word $s\in \AAA^\ell$ defines a cylinder $C_s \subset \AAA^{\N}$, the set of all words in $\AAA^{\N}$ beginning with $s$, and that, by definition of the product probability measure $\mu$,
\[
\m(C_s) = \prod_{j=0}^{\ell-1}\, p_{s_j}.
\]
Hence for the survival probability $p_n$ we have
\[
p_n = \m(S_n) = \sum_{s \in \Sigma_{n+r}}\, \m(C_s) = \sum_{s \in \Sigma_{n+r}}\, \prod_{j=0}^{n+r-1}\, p_{s_j},
\]
which becomes
\begin{equation}\label{form-p-2}
p_n = \m(S_n) = \sum_{s \in \Sigma_{n+r}}\, \prod_{a\in \AAA}\, ( p_a )^{N_s(a)}
\end{equation}
if we use the counting function of Definition \ref{count-fun}.

A classical method to study the exponential behaviour of a sequence is to use its generating function. Let $P(z)$ be the generating function of $\{p_n\}$, that is
\[
P(z) := \sum_{n=0}^\infty\, p_n\, z^n,
\]
and $\rho$ be the radius of convergence of $P(z)$. Then $\rho = e^{\escrate{w}^-}\le e^{\escrate{w}^+}$, thus $\escrate{w}:=\escrate{w}^-$ is given by the logarithm of the modulus of the smallest pole of $P(z)$.

We now use \eqref{form-p-2} to find an explicit expression for the generating function $P(z)$. For the sets $\Sigma_\ell$ defined in \eqref{sigmas}, let
\begin{equation}\label{gen-funct-sigma0}
\alpha_{k_1,\dots,k_A,\ell} := \# \set{ s \in \Sigma_\ell\, :\, N_s(a_i) = k_i\, ,\ i=1,\dots,A}.
\end{equation}
The power series
\begin{equation}\label{gen-funct-sigma}
\Sigma(x_{a_1},x_{a_2},\dots,x_{a_A},z) := \sum_{\ell=0}^\infty\, \Big( \sum_{k_1+\dots+k_A=\ell}\, \alpha_{k_1,\dots,k_A,\ell} \, (x_{a_1})^{k_1}\dots (x_{a_A})^{k_A}\Big) \, z^\ell,
\end{equation}
where all $k_i$ are assumed to be non-negative integers, is called the generating function of the sets $\Sigma_\ell$. We recall that
\[
\sum_{k_1+\dots+k_A=\ell}\, \alpha_{k_1,\dots,k_A,\ell} \, (x_{a_1})^{k_1}\dots (x_{a_A})^{k_A} = (x_{a_1}+x_{a_2}+\dots+x_{a_A})^\ell
\]
for all $\ell=0,\dots,r-1$.

Following \cite[Proposition I.4]{FS}, we prove the following result.
\begin{lemma}\label{res-fund}
Let $c_w = c_w(x_{a_1},x_{a_2},\dots,x_{a_A},z)$ be the weighted autocorrelation polynomial of the word $w$ of length $r$. Then
\[
\begin{aligned}
& \Sigma(x_{a_1},x_{a_2},\dots,x_{a_A},z)\\ & = \frac{c_w(x_{a_1},x_{a_2},\dots,x_{a_A},z)}{\Big( \prod_{i=1}^A\, (x_{a_i})^{N_w(a_i)} \Big)\, z^r + \Big[ 1- \Big( \sum_{i=1}^A\, x_{a_i}\Big)\, z\Big]\, c_w(x_{a_1},x_{a_2},\dots,x_{a_A},z)}\, .
\end{aligned}
\]
\end{lemma}

\begin{proof}
For $\ell \ge r$, set
\begin{equation} \label{doublius}
W_\ell := \set{s \in \AAA^\ell \, : \begin{array}{l}
\text{$s_{\ell-r}\, s_{\ell-r+1}\dots s_{\ell-1} = w$ and $s$ does not}\\ \text{contain $w$ in other positions}
\end{array}
}
\end{equation}
and denote by $W(x_{a_1},x_{a_2},\dots,x_{a_A},z)$ be the generating function of the sets $W_\ell$, defined as in \eqref{gen-funct-sigma0}-\eqref{gen-funct-sigma} with $W_\ell$ in lieu of $\Sigma_\ell$.

Let $s \in \Sigma := \cup_{\ell\ge 0} \Sigma_\ell$. By appending a letter to $s$, we obtain a non-empty word either in $\Sigma$ or in $W:=\cup_{\ell\ge r} W_\ell$. Hence the corresponding generating functions satisfy the equation
\begin{equation}\label{fine-1}
\begin{aligned}
& 1+\Big( \sum_{i=1}^A\, x_{a_i}\Big)\, z\, \Sigma(x_{a_1},x_{a_2},\dots,x_{a_A},z)\\ &= \Sigma(x_{a_1},x_{a_2},\dots,x_{a_A},z) + W(x_{a_1},x_{a_2},\dots,x_{a_A},z).
\end{aligned}
\end{equation}
Next, appending the word $w$ to a word $s \in \Sigma$, we obtain either a word in $W$ or a word with two appearances of $w$ in the last $2r-1$ symbols. The latter case occurs if one of the symbols $c_1,\dots,c_{r-1}$ of the autocorrelation vector of $w$ is equal to 1. If $c_i=1$ for some $i\ge 1$ then $(w_i\, w_{i+1}\dots w_{r-1}) = (w_0\, w_1 \dots w_{r-i-1})$, thus, appending $w$ to a word in $\Sigma$ whose last symbols are $(w_0\, w_1\dots w_{i-1})$, we get a word which can be written as a word in $W$ with $(w_{r-i}\, w_{r-i+1}\dots w_{r-1})$ appended at the end. Therefore we obtain a word with two appearances of $w$.

Hence we have the equation
\begin{equation}\label{fine-2}
\begin{aligned}
& \Big( \prod_{i=1}^A\, (x_{a_i})^{N_w(a_i)} \Big)\, z^r\, \Sigma(x_{a_1},x_{a_2},\dots,x_{a_A},z)\\ & = W(x_{a_1},x_{a_2},\dots,x_{a_A},z)\, c_w(x_{a_1},x_{a_2},\dots,x_{a_A},z).
\end{aligned}
\end{equation}
The result follows by solving \eqref{fine-1} and \eqref{fine-2}.
\end{proof}

From \eqref{form-p-2} it follows that\medskip
\[
\begin{aligned}
P(z) =\ & \sum_{n=0}^\infty\, \Big(\sum_{s \in \Sigma_{n+r}}\, \prod_{i=1}^A\, ( p_{a_i} )^{N_s(a_i)} \Big)\, z^n\\ =\ & \sum_{n=0}^\infty\, \Big(\sum_{k_1+\dots+k_A=n+r}\, \alpha_{k_1,\dots,k_A,n+r} \, (p_{a_1})^{k_1}\dots (p_{a_A})^{k_A}\Big)\, z^n\\
=\ & z^{-r}\, \sum_{\ell=r}^\infty\, \Big(\sum_{k_1+\dots+k_A=\ell}\, \alpha_{k_1,\dots,k_A,\ell} \, (p_{a_1})^{k_1}\dots (p_{a_A})^{k_A}\Big) \, z^\ell,
\end{aligned}
\]
hence finally
\begin{equation}\label{definitiva}
P(z) = z^{-r}\, \Big( \Sigma(p_{a_1},p_{a_2},\dots,p_{a_A},z) - \sum_{\ell=0}^{r-1}\, z^\ell\Big).
\end{equation}
Using Lemma \ref{res-fund}, it follows that $P(z)$ is a rational function, with smallest positive pole given by the smallest positive zero of the denominator of $\Sigma(p_{a_1},p_{a_2},\dots,p_{a_A},z)$. Hence the proof of Proposition \ref{prop:escape-rate} is finished.

\section{Proof of Proposition \ref{prop:escape-rate-markov}} \label{app:proof-thm-markov}

We argue as in Appendix \ref{app:proof-thm-iid}. Given a hole $w\in \AAA^{*}_{_\Pi}$ of length $r$ we have by \eqref{sigmas} and \eqref{esses}
\[
p_n = m_{_\Pi}(S_n) = m_{_\Pi} \Big(\set{ \omega\in \AAA^{\N}_{_\Pi}\, :\, \omega^{n+r}\, \in \Sigma_{n+r}}\Big)
\]
for the survival probability. Since, for a given $s=(s_0\, \dots\, s_{\ell-1})\in \{a,b\}^\ell$,
\[
m_{_\Pi}(C_s) = p_{s_0}\, \prod_{j=0}^{\ell-2}\, \pi_{s_j\, s_{j+1}}
\]
where $p=(p_a,p_b)$ is the vector defined in Section \ref{sec:markov}, we find
\[
p_n = \sum_{s \in \Sigma_{n+r}}\, m_{_\Pi}(C_s) = \sum_{s \in \Sigma_{n+r}}\, p_{s_0}\, \prod_{j=0}^{n+r-2}\, \pi_{s_j\, s_{j+1}}.
\]
Using the notation $\mathbf{k}=(k_1,k_2,k_3,k_4)$ and $|\mathbf{k}|:= k_1+k_2+k_3+k_4$, we can then write
\begin{equation}\label{form-p-markov}
p_n = \sum_{|\mathbf{k}|=n+r-1}\, \Big( \alpha^a_{\mathbf{k}}\, p_a\, \pi_{aa}^{k_1}\,\pi_{ab}^{k_2}\,\pi_{ba}^{k_3}\,\pi_{bb}^{k_4} + \alpha^b_{\mathbf{k}}\, p_b\, \pi_{aa}^{k_1}\,\pi_{ab}^{k_2}\,\pi_{ba}^{k_3}\,\pi_{bb}^{k_4} \Big)
\end{equation}
where $\alpha^a_{\mathbf{k}}$ and $\alpha^b_{\mathbf{k}}$ denote the number of words $s \in \Sigma_{n+r}$ which begin with $a$ and $b$, respectively, and for which $\mathbf{k}$ is the vector of the number of appearances of $\pi_{aa},\pi_{ab},\pi_{ba},\pi_{bb}$, respectively, in the measure of the cylinder $C_s$.

In order to obtain a connection with the survival probabilities as expressed in \eqref{form-p-markov}, we need to write the generating function of the sets $\Sigma_\ell$ taking into account the number of transitions from a letter to another, and the first letter of a word. We write
\begin{equation}\label{gf-markov}
\begin{aligned}
& \Sigma(x_a,x_b,y_{aa},y_{ab},y_{ba},y_{bb},z)\\ & := 1 + \Sigma^a(x_a,y_{aa},y_{ab},y_{ba},y_{bb},z) + \Sigma^b(x_b,y_{aa},y_{ab},y_{ba},y_{bb},z)
\end{aligned}
\end{equation}
where
\[
\begin{aligned}
& \Sigma^a(x_a,x_b,y_{aa},y_{ab},y_{ba},y_{bb},z)\\ & := \sum_{\ell=1}^\infty\, \Big[ \sum_{|\mathbf{k}|=\ell-1}\, \Big(\alpha^{a,a}_{\mathbf{k}}\, x_a\, y_{aa}^{k_1}\,y_{ab}^{k_2}\,y_{ba}^{k_3}\,y_{bb}^{k_4}+\alpha^{b,a}_{\mathbf{k}}\, x_b\, y_{aa}^{k_1}\,y_{ab}^{k_2}\,y_{ba}^{k_3}\,y_{bb}^{k_4}\Big) \Big]\, z^\ell
\end{aligned}
\]
is the generating function of words of length at least 2 and with last letter equal to $a$. Here $\alpha^{a,a}_{\mathbf{k}}$ and $\alpha^{b,a}_{\mathbf{k}}$ denote the number of words of length $|\mathbf{k}|+1$ not containing $w$ as a pattern, which begin with $a$ and $b$, respectively, which end with $a$, and for which $\mathbf{k}$ is the vector of the number of appearances of the patterns $aa,ab,ba,bb$.

A similar formulation holds for $\Sigma^b(x_a,x_b,y_{aa},y_{ab},y_{ba},y_{bb},z)$, the generating function of words of length at least 2 and with last letter equal to $b$. With the above notations we have
\[
\alpha^{a,a}_{\mathbf{k}} + \alpha^{a,b}_{\mathbf{k}} = \alpha^{a}_{\mathbf{k}}\, , \qquad \alpha^{b,a}_{\mathbf{k}} + \alpha^{b,b}_{\mathbf{k}} = \alpha^{b}_{\mathbf{k}}.
\]
The analog of \eqref{definitiva} then follows for the generating function of $\{p_n\}$
\begin{equation}\label{definitiva-markov}
P(z) = \sum_{n\ge 0}\, p_n\, z^n = z^{-r}\, \Big( \Sigma(p_a,p_b,\pi_{aa},\pi_{ab},\pi_{ba},\pi_{bb},z) - q(z) \Big)
\end{equation}
where $q(z)$ is a polynomial of degree $r-1$ in $z$. As in Appendix \ref{app:proof-thm-iid}, the proof of Proposition \ref{prop:escape-rate-markov} ends by showing that $\Sigma(x_a,x_b,y_{aa},y_{ab},y_{ba},y_{bb},z) $ is a rational function and finding its denominator.

We repeat the proof of Lemma \ref{res-fund} to prove
\begin{lemma}\label{res-funs-markov}
Let $c_{w,M}(y_{aa},y_{ab},y_{ba},y_{bb},z)$ be the Markovian weighted autocorrelation polynomial of a word $w$ of length $r$. Then the generating function \eqref{gf-markov} is a rational function with denominator given by
\[
\begin{aligned}
& y_{w_0 w_1}\dots y_{w_{r-2} w_{r-1}}\, z^r\, (y_{w_{r-1}w_0} - \chi\, z\, \delta_{w_0,w_{r-1}})\\ &+ \Big(1-(y_{aa}+y_{bb})\, z +\chi \, z^2 \Big)\, c_{w,M}(y_{aa},y_{ab},y_{ba},y_{bb},z),
\end{aligned}
\]
where $\chi:= y_{aa}y_{bb}-y_{ab}y_{ba}$.
\end{lemma}

\begin{proof}
For $\ell\ge r$, let $W_\ell$ as in \eqref{doublius}, and $W:=\cup_{\ell\ge r} W_\ell$. Also we use the notation $W(x_a,x_b,y_{aa},y_{ab},y_{ba},y_{bb},z)$ for the generating function of the sets $W_\ell$.

Let $s\in \Sigma = \cup_{\ell\ge 0}\, \Sigma_\ell$. By appending a letter to $s$, we obtain a non-empty word in $\Sigma$ or in $W$, but now we need to distinguish between the cases where we append the letter $a$ or $b$. We find the following equations for the generating functions defined in \eqref{gf-markov} (we drop the dependence on most of the variables for reasons of readability). If we append the letter $a$, we might find a word in $W$ only if $w_{r-1}=a$, and we do not find it ending with $b$, hence
\begin{equation}\label{fine-app-a}
x_a\, z + \Sigma^a(z)\, y_{aa}\, z + \Sigma^b(z) \, y_{ba}\, z = \Sigma^a(z) + W(z)\, \delta_{a, w_{r-1}}.
\end{equation}
Analogously,
\begin{equation}\label{fine-app-b}
x_b\, z + \Sigma^a(z)\, y_{ab}\, z + \Sigma^b(z) \, y_{bb}\, z = \Sigma^b(z) + W(z)\, \delta_{b, w_{r-1}}.
\end{equation}
Next, appending the word $w$ to a word $s \in \Sigma$, we obtain either a word in $W$ or a word with two appearances of $w$ in the last $2r-1$ symbols, and again this is regulated by the Markovian weighted autocorrelation polynomial of $w$. We find the equation
\begin{equation}\label{fine-app-w}
\begin{aligned}
& x_{w_0}\, y_{w_0 w_1}\dots y_{w_{r-2} w_{r-1}} + \Big( \Sigma^a(z)\, z^r\, y_{a w_0} + \Sigma^b(z)\, z^r\, y_{b w_0} \Big)\, y_{w_0 w_1}\dots y_{w_{r-2} w_{r-1}}\\ &= W(z)\, c_{w,M}(z).
\end{aligned}
\end{equation}
Solving the system given by \eqref{fine-app-a},\eqref{fine-app-b},\eqref{fine-app-w} for $\Sigma^a$, $\Sigma^b$ and $W$ in the four possible cases for $w_0$ and $w_{r-1}$, we find that the solution of the systems is such that $\Sigma(x_a,x_b,y_{aa},y_{ab},y_{ba},y_{bb},z)$ is a rational function with denominator given by
\[
\begin{aligned}
& y_{w_0 w_1}\dots y_{w_{r-2} w_{r-1}}\, z^r\, \Big(y_{w_{r-1}w_0} - (y_{aa}y_{bb}-y_{ab}y_{ba})\, z\, \delta_{w_0,w_{r-1}}\Big)\\ & + \Big(1-(y_{aa}+y_{bb})\, z +(y_{aa}y_{bb}-y_{ab}y_{ba})\, z^2 \Big)\, c_{w,M}(z).
\end{aligned}
\]
The proposition is proved.
\end{proof}

Using now \eqref{definitiva-markov}, with the relations $\pi_{aa}+\pi_{ab}=\pi_{ba}+\pi_{bb}=1$, we find that the denominator of $P(z)$ is given by the polynomial $\tau_{w,\Pi}$ in \eqref{pol-tau-markov}.


\end{document}